\documentclass[12pt]{article}
\usepackage{authblk}
\usepackage{amsthm}
\usepackage{amssymb}
\usepackage{amsmath}
\usepackage{amsfonts}
\usepackage[margin=1.2in]{geometry}    
\usepackage{hyperref} 
\usepackage{cleveref} 
\usepackage[toc,page]{appendix} 
\usepackage{thmtools, thm-restate} 
\usepackage{longtable} 
\usepackage{esvect} 
\usepackage{accents} 
\usepackage[all]{xy} 

\usepackage{titlefoot} 

\usepackage{makeidx}
\makeindex

\newcommand{\QQ}{\mathbb{Q}}
\newcommand{\NN}{\mathbb{N}}
\newcommand{\RR}{\mathbb{R}}
\newcommand{\BB}{\mathbb{B}}
\newcommand{\ZZ}{\mathbb{Z}}
\newcommand{\EE}{\mathbb{E}}
\newcommand{\AAA}{\mathcal{A}}
\newcommand{\DD}{\mathcal{D}}
\newcommand{\MM}{\mathcal{M}}
\newcommand{\PP}{\mathbb{P}}
\newcommand{\JJ}{\mathbb{J}}

\newcommand{\II}{\mathcal{I}}
\newcommand{\HH}{\mathcal{H}}
\newcommand{\NMM}{\mathcal{M}^*}

\newcommand{\similar}{similar }

\newcommand{\partition}{\gamma}

\newcommand{\Lambdap}{\Lambda_p}
\newcommand{\RightShift}{{\boldsymbol{\sigma}_{-}}}

\newcommand{\bpn}{)_{p^n}}
\newcommand{\bp}{)_{p}}

\newcommand{\Shift}{\boldsymbol{\sigma}}

\newcommand{\deriv}{\varphi}
\newcommand{\restrict}{\big |}

\newcommand{\DR}{\ZZ \left[ \frac{1}{p} \right] }

\newcommand{\twopartdef}[4]
{
        \left\{
                \begin{array}{ll}
                        #1 & \mbox{if } #2 \\
                        #3 & \mbox{if } #4
                \end{array}
        \right.
}

\newcommand{\twopartdefotherwise}[3]
{
        \left\{
                \begin{array}{ll}
                        #1 & \mbox{if } #2 \\
                        #3 & \mbox{otherwise }
                \end{array}
        \right.
}

\newcommand{\fullSeq}[1]{\boldsymbol{#1}}
\newcommand{\minusSeq}[1]{\boldsymbol{#1}_{-}}

\declaretheorem[numberwithin=section]{theorem}
\declaretheorem[sibling=theorem]{lemma}
\declaretheorem[sibling=theorem]{definition}
\declaretheorem[sibling=theorem]{proposition}

\declaretheorem[sibling=theorem]{remark}
\declaretheorem[sibling=theorem]{corollary}

\newcommand {\PlainXZ}{X_{\ZZ}}
\newcommand {\PlainXZtilde}{\tilde{X}_{\ZZ}}

\newcommand{\Mext}{M_p}
\newcommand{\Pext}{\tilde{Q}_{\ZZ}}
\newcommand{\Pchain}{Q_{\ZZ}}
\newcommand{\Pminus}{Q}
\newcommand{\Ptilde}{\tilde{Q}}
\newcommand{\PM}{\hat{\nu}}

\numberwithin{equation}{section}
\begin{document}

\title{Spatial Recurrence for Ergodic Fractal Measures}
\author{Nadav Dym }
\affil{Hebrew University of Jerusalem\\ (email: nadav.dym@weizmann.ac.il)}
\renewcommand\Affilfont{\itshape\small}
\date{}
\maketitle
\begin{abstract}
We discuss an invertible version of Furstenberg's `Ergodic CP Shift Systems'.  We show that  the explicit regularity of these dynamical systems with respect to magnification of measures, implies certain regularity with respect to translation of measures; We show that the translation action on measures is non-singular, and prove pointwise discrete and continuous ergodic theorems for the translation action.

\end{abstract}

\unmarkedfntext{This research was supported by ERC grant 306494.}
\section{Introduction}
CP Distributions were  introduced by Furstenberg in \cite{furst} as a tool in his proof of dimension conservation for homogeneous fractals, though they also appear in slightly different form in earlier work \cite{furstAncient,furstANDweiss}.
CP\ distributions are probability measures $Q$ defined on a set of points of the form $(\mu,x)$, which are invariant with respect to an operation of `magnifying' $\mu$ around a neighborhood of $x$. When the distribution is also ergodic with respect to the magnifying operation, we say that it is an Ergodic CP Distribution (ECPD).

CP distributions and their equivalents have played an important role in many recent works on fractal geometry \cite{furst,local,equidistribution,KSS,dimension}. They often arise in explicit constructions, like self-similar measures and random fractal constructions, but there are also many examples in which CP distributions arise implicitly, for instance by a limiting process in which we start with an arbitrary measure and `zoom in' to a point, or as the `slice' measures (conditional measures on affine subspaces) of higher-dimensional CP distributions (see \cite{furst,hochman}). Because an explicit description of the distribution is often not available, it becomes important to better understand the properties of CP distributions in general.

 In this paper we discuss Extended ECPD which are an invertible version of ECPD (see \cite{hochman}). Our goal is to show that the explicit requirements of regularity with respect to the magnification operation on measures, which appear in the definition of Extended ECPD, in fact imply certain regularity also with respect to a `translation and normalization' operation on measures. We exclude from our discussion  Extended ECPD which are trivial in the sense that typical measures are Dirac measures. Similar questions were addressed by Host in \cite{host}, where he proves that a `non deterministic' measure on $[0,1)$ invariant and ergodic under the  $\times p $ map (which we think of as a magnification map in this context) is conservative with respect to the action of translation by numbers whose base $p$ representation is finite, and by Medynets and Solomyak \cite{solomyak} who proved a second-order ergodic theorem for the action of translation by $\RR^d$ on a different class of  `fractal dynamical systems' (self-similar tiling systems).

Using similar methodology to Host's, we prove that the translation and normalization action on measures is conservative, and use this conservativity to define a modified translation and normalization action which is non-singular. This in turn enables the application of the ergodic theorem for non-singular transformations, which shows that the discrete ergodic averages
$$\frac{1}{\nu[0,N)}\sum_{n=0}^{N-1} \nu[n,n+1) \cdot f(t_n^*\nu)  $$
and the continuous ergodic averages
$$\frac{1}{\nu[0,T)}\int_0^T f(t_x^*\nu)d\nu(x) $$
(where $t_x^*\nu $ is the measure $\nu $ translated by $x$ and then normalized) converge for typical $\nu$.

We now begin Subsection~\ref{sub:term} in which we present some terminology we use throughout the paper. We then define Extended ECPD in Subsection~\ref{sub:extended} , which will enable us to give a complete presentation of our main results in Subsection~\ref{sub:MainResults}. We will then give the  outline of the rest of the paper.

%

\textbf{Acknowledgments}
This paper is an adaptation of the  Master's thesis of the author submitted to the Hebrew University of Jerusalem.  I would like to thank my advisor  Michael Hochman for providing me with interesting questions to work on, for many helpful discussions and suggestions, and in general, for his substantial help in making the work on this project a pleasant experience.

\pagebreak
\textbf{List of Notation}
\begin{longtable}[l]{p{100pt} p{300pt} }
 $\MM(\RR)$      & Radon positive measures on $\RR$.  \\
 $N$         &   Map normalizing elements of $\MM(\RR)$.        \\
  $\NMM$         & Space of normalized measures.         \\
  $\PP_0[0,1)$ & Space of probability measures on $[0,1)$, and zero.\\
  $\Lambdap$   & The set $\{0,1,\ldots,p-1 \} $. \\
  $\fullSeq{i},\fullSeq{\mu} $ & Sequences of the form $\fullSeq{i}=(i_n)_{n \in \ZZ}$ and $\fullSeq{\mu}=(\mu_n)_{n \in \ZZ} $.\\
  $\minusSeq{i},\minusSeq{\mu} $ & Sequences of the form $\minusSeq{i}=(i_n)_{n \leq 0}$ and $\minusSeq{\mu}=(\mu_n)_{n \leq 0} $.\\
  $\PlainXZ, \PlainXZtilde $ & Spaces with points of the form $(\fullSeq{\mu},\fullSeq{i}) $, $(\nu,\fullSeq{i}). $\\
  $X, \tilde{X} $ & Spaces with points of the form $(\minusSeq{\mu},\minusSeq{i}) $, $(\nu,\minusSeq{i}). $\\
  $\Mext $ & Zooming in map on Extended ECPS.\\
  $S_a $, $T_k $ & Translation maps on $X$ and $\tilde{X}$.\\
  $\Phi$, $\theta$ & Isomorphisms $\Phi:\PlainXZ \to \PlainXZtilde $ and $\theta:X \to \tilde{X} $.  \\
  $\hat{\nu}$,$\tilde{\pi}_{\MM} $, $\hat{\mu}_n$ & Projections to the measure coordinate.\\
  $\hat{i}_n $ & Projection to the symbolic coordinate.\\
  $t_x$ & Translation by $x$ on $\RR $, or the induce map on $\MM(\RR) $.\\
  $t_x^* $ & Translation and normalization map on $\NMM $.\\
  $\Shift $,$\RightShift $ & Left shift, right shift.\\
  $E_j,E_j^{-} $ & Sequences $\fullSeq{i}$ or $\minusSeq{i}$ with $i_n=j $ for all negative enough $n$.    \\
  $D_{p^n} $, $D$ & Points of the form $kp^{-n}$, and the union of $D_{p^n} $.\\
  $G_n,G $ & Points of the form $(\ldots,0,0,i_n,\ldots,i_{-1},i_0)$, and the union of $G_n$.\\
  $[i_1,\ldots,i_n \bpn$ & The interval $[\sum_{j=1}^n i_j p^{-j},\sum_{j=1}^n i_j p^{-j}+p^{-n}) $.
  \end{longtable} 
\subsection{Terminology} \label{sub:term}
\begin{enumerate}
\item We  denote sequences $(x_n)_{n \in \ZZ} $ by $\fullSeq{x}$, and sequences $(x_n)_{n \leq 0} $ by $\minusSeq{x}$. We use $x_j^l $ as a shortened notation for the subsequence $(x_j,x_{j+1}, \ldots, x_{l}) $.
\item For every space of the form $X^{\ZZ}$, $\Shift$ will denote the left-shift operator, defined by $$(\Shift(\fullSeq{x}))_n=x_{n+1}$$
and $ \Shift_{-} $ will denote
the right-shift operator on
$ X^{\ZZ} $ or $X^{ \ZZ_{-} }$ defined by
$$(\Shift_{-}(\fullSeq{x}))_n=x_{n-1} \text{ or } (\Shift_{-}(\minusSeq{x}))_n=x_{n-1}   $$
\item All the spaces we discuss are separable metric spaces, and we always take the Borel $\sigma$-algebra on these spaces, which  we will denote  by $\BB$.
\item Recall that if $(X,\BB,\mu)$ is a measure space,  $(Y,\AAA) $ is a measurable space, and $\rho:X \rightarrow Y$ is a measurable function, then \emph{ the pushforward of $\mu$ by $\rho$ } is the measure $\nu$ on $Y$ defined by
$$\int f(y) d\nu(y)=\int f \circ \rho (x) d\mu(x) $$
We will denote this measure by $\rho \mu $ or $d\rho \mu  $.

If $g:X \to X $ is a non-negative measurable function, then the \emph{multiplication of $\mu$ by $g$}, is the measure $\nu $ on $X$, defined by
$$\int f(x)d\nu(x)=\int f(x) \cdot g(x) d\mu(x) $$
We denote this measure by $gd\mu $.
\item If $(X,\BB,\mu)$ is a probability space, $\AAA \subseteq \BB $ is a sub-$\sigma$-algebra, and $B \in \BB$, then $\PP_{\mu}\left(B|\AAA  \right) $ is the function $\EE_{\mu}\left(1_B| \AAA   \right) $.

For measurable functions $Z_1,\ldots,Z_d:(X,\BB) \to (Y,\AAA) $,  we define
$$\PP_{\mu}\left( B |Z_1^d    \right)=\PP_{\mu}\left( B| \sigma(Z_1^{-1}\AAA,\ldots,Z_d^{-1}\AAA) \right) $$
\item We will usually denote measures on $\RR $ by $\nu $, and measures supported in $[0,1) $ by $\mu $.  To avoid confusion, we call measures on a space of measures distributions, and denote them by $P,Q $ etc.

\end{enumerate}
\subsection{Extended ECPS}\label{sub:extended}
In the following we will define Extended ECPS (ergodic CP systems). Our notation and definitions resemble Hochman's \cite{hochman}, but are not identical. A comparison is presented in Appendix~\ref{Ap:Notation}.

The measures we will perform our `zooming in' on will be members of $\MM(\RR),$  the space of positive Radon measures on $\RR$ endowed  with the weak topology (this is a metrizable, separable topology in which $\mu_n \rightarrow \mu$ if and only if for every $f \in C_c(\RR)$, $\int f d\mu_n \rightarrow \int f d\mu $. For more details see \cite{mattila}). The following terminology will be helpful for the construction we will soon describe:
\begin{definition}
The restriction of $\nu\in \MM(\RR)$ to a Borel set $A \subseteq \RR$ will be the measure $1_A d\nu \in \MM(\RR)$.
\end{definition}
\begin{definition}
$\rho:\RR \to \RR$ is an orientation preserving homothety if $\rho(x)=ax+b $ where $a>0,b\in \RR$.
\end{definition}
\begin{definition}
We say that $\nu_1$ and $\nu_2$ are similar, if there is an orientation preserving homothety $\rho$ and $\lambda>0$ such that $d\nu_1=\lambda d\rho \nu_2$.  
\end{definition}

 \textbf{Normalization of Measures. } Normally when defining Extended ECPS, the discussion is restricted to measures $\mu$ which give $[0,1) $ positive measure, and these measures are assumed to be normalized so that $\mu[0,1)=1$. The structure of Extended ECPS then promises that `zooming in' to $\mu $ gives rise to a new measure which also gives $[0,1) $ positive mass, and hence this new measure can also be normalized so that its restriction to $[0,1) $ is a probability measure.
  For our purposes here, however, it will be useful to extend this normalization of measures also to measures $\mu $ with $\mu[0,1)=0 $, since once we allow typical measures of Extended ECPS to be translated, it is no longer guaranteed that the translated measure will be normalizable in the usual sense. We thus define a normalizing map in the following way
\begin{definition}
For every Radon measure $\nu \neq 0$ on $\RR$, we define
$$\psi(\nu)=\min\{n\in \NN:\nu[-(n-1),n)>0\} \quad (\NN=1,2,\ldots) $$
For every $\nu \in \MM(\RR)$ we define
$$N \nu=\twopartdef{\frac{\nu}{\nu[-(\psi(\nu)-1),\psi(\nu))}}{\nu \neq 0}{0}{\nu=0}$$
\end{definition}
We note that if $\mu[0,1)>0 $, this normalization coincides with the standard normalization described above.  This normalization does not seem very natural, and indeed it is not unique in the sense that different extensions of the standard normalization which fulfill the conditions of Remark~\ref{rem:Norm} can be used to achieve the same results which we describe below. However while we use this normalization for proving our results, they can in fact be stated in a way which is independent of the normalization, as we will discuss at the end of the introduction.

Having defined the normalizing map, we will now restrict our discussion to the space of normalized measures $\NMM $, which is just the image of $\MM(\RR)$ under $N$.

Now fix $p \in \NN \setminus \{1\} $, and denote $\Lambdap=\{0,1,\ldots,p-1 \} $.
The interval $[0,1)$ can be divided into $p$ intervals of the form $[\frac{i}{p},\frac{i+1}{p}) $. To simplify notation, for $i \in \Lambdap$ we will write $[i \bp$ instead of  $[\frac{i}{p},\frac{i+1}{p}) $. Similarly, for any $n \in \NN$, $[0,1) $ can be divided into $p^n$ intervals of equal length, and for $(i_1,i_2, \ldots, i_n) \in \Lambdap^n$ we write
$$[i_1,i_2, \ldots,i_n \bpn=\bigg[\sum_{k=1}^{n}\frac{i_k}{p^k},\sum_{k=1}^{n}\frac{i_k}{p^k}+\frac{1}{p^n} \bigg) $$

We also define $\rho_i$ to be the orientation preserving homothety that takes $[i \bp $ to $[0,1)$, i.e.
$$\rho_i(x)=px-i$$

 We now  consider the space $\NMM \times \Lambdap^{\ZZ}$ obtained by adding indices to $\NMM$. On this space we define the projection maps
$$\hat{\nu}(\nu,\fullSeq{i})=\nu $$
$$\hat{i}_n(\nu,\fullSeq{i})=i_n$$

We can now define our `zooming in' map
\begin{definition}
The `zooming in' map  $\Mext:\NMM \times \Lambdap^{\ZZ} \to \NMM \times \Lambdap^{\ZZ} $ will be the (invertible) map
$$\Mext(\nu,\fullSeq{i})=(N\rho_{i_1}\nu,\Shift(\fullSeq{i}))$$
\end{definition}

Note that the restriction of $ N\rho_{i_1}\nu$ to $[0,1)$ is a probability measure (or zero) which is \similar to the restriction of $\nu$ to $[i_1)_{(p)}$. More generally, the measure component of $(\Mext)^n(\nu,\fullSeq{i}) $ will be a measure whose restriction to $[0,1)$ is a probability measure (or zero) which is \similar to the restriction of $\nu$ to $[i_1,i_2,\ldots,i_n \bpn $.

We can now define an Extended ECPS 
\vspace{0.5 cm}
\begin{definition}
A probability distribution $\Pext$ on $(\NMM \times \Lambdap^{\ZZ},\BB)$ is called an \emph{adapted distribution}, if for every $j \in \Lambdap $ and $\Pext$ almost every $(\nu,\fullSeq{i}) $,
$$\PP_{\Pext}\left(\hat{i}_1=j|(\hat{\nu},(\hat{i}_n)_{n \leq 0})\right)(\nu,\fullSeq{i})=\nu[j \bp $$
\end{definition}
\vspace{0.5 cm}
In other words, the probability of `zooming in' to $\nu \restrict_{[j\bp}$ is exactly $\nu [j \bp$. Summing over $j=0,\ldots,p-1$ shows that when $\Pext $ is an adapted distributions, a.e. $\nu$ has $\nu[0,1)=1 $.
\vspace{0.5 cm}
\begin{definition}
A probability distribution $\Pext$ on $(\NMM \times \Lambdap^{\ZZ},\BB)$ is called an \emph{Extended Ergodic CP Distribution} (Extended ECPD) if it is invariant and ergodic with respect to $\Mext$, and it is adapted.\\ If $\Pext $ is an Extended ECPD,
$(\NMM \times \Lambdap^{\ZZ},\Mext,\Pext)$ is called an \emph{Extended Ergodic CP system} (Extended ECPS).
\end{definition}
\paragraph{Non-Deterministic Distributions}
It will be convenient to assume that the Extended ECPD we discuss are non-deterministic:
\begin{definition}  We say that a distribution $\Pext $ on $\NMM \times \Lambdap^{\ZZ} $ is deterministic, if for $\Pext$ almost every $(\nu,\fullSeq{i}) $,
\begin{equation}\label{eq:determinism}
\PP_{\Pext}\left( \hat{i}_1=i_1|(\hat{\nu},(\hat{i}_n)_{n \leq 0})\right)\left(\nu,\fullSeq{i}\right)=1
\end{equation}
Otherwise we say that the distribution is non-deterministic.
\end{definition}
If an Extended ECPD is deterministic then it follows from the $\Mext$ invariance of   $\Pext$ that a.e. measure $\nu$ is a Dirac measure supported on a point in $[0,1)$. We restrict our discussion to the non-degenerate case of non-deterministic distributions.

 If $\Pext$ is a non-deterministic Extended ECPD, then for every $j \in \Lambdap$ and almost every $(\nu,\fullSeq{i})$, the index sequence $\fullSeq{i}$ isn't in the set
$$E_j=\{\fullSeq{i}: \; i_n=j \text{ for all negative enough } n \} $$
 This again is due to $\Mext$ invariance. Thus we can consider non-deterministic Extended ECPD to be defined only on the $\Mext$ invariant set
$$\PlainXZtilde=\{(\nu,\fullSeq{i}): \; \; \nu \in \NMM, \text{ and  } \forall j\in \Lambdap, \; \fullSeq{i} \not \in E_j \} $$

\subsection{ Results}\label{sub:MainResults}

As we discussed earlier, we are interested in the behavior of the measure component of Extended ECPS under translations. We will now define this more precisely:
\begin{definition}
For every $x \in \RR$, let $t_x:\RR \to \RR$ be the map
$$t_x(y)=y-x$$
$t_x \nu $ will be the pushforward of $\nu$ by $t_x$. We also define $t^*_x \nu=Nt_x\nu$.
\end{definition}
The maps $\{ t^*_x\}_{x \in \RR}$ define a translation-normalization action of $\RR$ on $\NMM$, and the pushforward of an extended ECPD $\Pext$ by the projection $\PM $ induces a distribution $\PM \Pext $ on $\NMM$.
\begin{definition}\label{def:cons}
Let $G $ be a group which acts measurably on a Borel probability space $(\Omega, \BB, \mu )$, where $\Omega$ is a separable metric space.
\begin{enumerate}
\item We say that the action of $G$ on $\Omega$ is conservative with respect to $\mu$, if for every $A \in \BB $ with $\mu(A)>0 $, there is a $g \in G \setminus \{ 1_G \} $, such that $\mu(A \cap gA)>0 $.
\item We say that the action of $G$ on $\Omega$ is strictly singular with respect to $\mu$, if for all $g \in G \setminus \{ 1_G \} $, $g\mu \perp \mu $.
\item We say that the action of $G$ on $\Omega$ is recurrent with respect to $\mu$, if for $\mu$ almost every $x \in \Omega$, there is a sequence $(g_n)_{n \in \NN} \subseteq G \setminus \{1_G \} $, such that $g_nx \rightarrow x $.
\end{enumerate}
\end{definition}
Conservativity implies recurrence, and if $G$ is countable, then conservativity also implies that the action of $G$ on $\Omega$ isn't strictly singular with respect to $\mu$. However, recurrence, or failing to be strictly singular, does not necessarily imply conservativity. For example, the action of $\QQ$ on $(\RR,\BB,\frac{1}{2}(\delta_0+\delta_1))$ by addition is recurrent, is not strictly singular, and yet is not conservative.

We can now state our results.
\begin{theorem} \label{th:conservative}
If $\Pext$ is a non-deterministic ECPD, then the translation-normalization action of $\ZZ$ on $\NMM$ is conservative with respect to $\PM \Pext $.

\end{theorem}

As conservativity implies recurrence, one can immediately conclude
\begin{corollary}\label{cor:recurrence}
 If $\Pext$ is a non-deterministic ECPD, then for almost every $\nu$, there is a sequence $k_n \in \ZZ$ such that $t^*_{k_n}\nu \rightarrow \nu$ in the weak topology.
\end{corollary}

Using the fact that the translation action is conservative, we prove that
$$\tau(\nu)=\min\{n \in \NN: \; t_n\nu [0,1)>0\}$$
is finite for a.e. $\nu$, and thus the map
\begin{equation}\label{eq:inducedMap}
t(\nu)= t_{\tau(\nu)}^*\nu
\end{equation}
is well defined. In fact we show that this map is non-singular, and we can thus apply the ergodic theorem for non-singular transformations to obtain
 discrete and continuous Pointwise  Ergodic Theorems for non-deterministic ECPD:
\begin{theorem}\label{th:DiscreteErgodic}
If $\Pext$ is a non-deterministic ECPD, then for every $f \in L^1(\NMM,\PM\Pext)$ and $\PM\Pext$ almost every $\nu$ we have
$$\lim_{N} \frac{1}{\nu([0,N))}\sum_{n=0}^{N-1} \nu[n,n+1) \cdot f(t_n^*\nu) =\EE_{\PM\Pext}(f|\JJ)(\nu)$$
\end{theorem}
where $\JJ$ is the $\sigma$ algebra invariant under $t$.

\begin{theorem}\label{th:ContErgodic}
Let $\Pext$ be a non-deterministic ECPD. For every bounded measurable $f:\NMM\rightarrow \mathbb{C}$ define $F^f=\int_0^1 f(t_x^*\nu)d\nu(x)$, then for $\PM\Pext$ almost every $\nu$,
$$\lim \frac{1}{\nu[0,T)} \int_0^T f(t_x^*\nu)d\nu(x)=\EE_{\PM\Pext}(F^f|\JJ)(\nu)$$
\end{theorem}

We note that we will prove slightly stronger versions of the theorems stated above, since we will prove them for a factor $\tilde{X} $ of $\PlainXZtilde $ (which we will define later) which is larger than the factor $\NMM$.

\textbf{Independence of  normalization extension. } As mentioned above, our results can be phrased so that they are independent of the way the normalization map is extended to measures with $\mu[0,1)=0$. The action of $\{t_n^* \}_{n \in \ZZ} $ is conservative if and only if the action of $\{t^n \}_{n \in \ZZ} $ is conservative. Since the set of measures $\mu $ with $\mu[0,1)=1 $ is invariant under $t$, we see that Theorem~\ref{th:conservative}  does not in fact depend on the extension of the normalization to measures $\mu $ with $\mu[0,1)=0$.

An equivalent phrasing for the discrete and continuous ergodic theorems, is to remove the requirement that $f $ will be defined only on $\NMM $, and instead allow $f$ to be defined on all of $\MM(\RR)$, but require that for all $\lambda>0 $ and $\mu \in \MM(\RR) $, $f(\mu)=f(\lambda \mu) $. In this case $f(t_x^*\nu)=f(t_x \nu) $, and so again in this phrasing the results do not depend on the extension of the normalization.

\paragraph{Conservativity of typical measures}
 A stronger property than the conservativity of $\Pext $ with respect to translations of measures in $\NMM $, is the conservativity of typical $\nu $ with respect to the action of $\RR $ on itself by translation. Indeed we provide an example of an ECPD for which typical measures are strictly singular with respect to the translation action of $\RR $ on itself, although the ECPD itself is conservative with respect to $\ZZ $ translations.

In general we show that conservativity of typical measures holds, with respect to $\DR $ translations, if and only if a property we call bilateral determinism holds. If bilateral determinism doesn't hold, then typical $\nu $ are strictly singular with respect to $\DR $ translations. In particular, for any Extended ECPD, either a.e. measure $\nu$ is conservative, or  a.e. measure is strictly singular.

Unlike deterministic ECPD,  bilaterally deterministic ECPD are `non-trivial'. We show this by showing that the bilaterally deterministic symbolic measures of  \cite{weissANDornstein} generate bilaterally deterministic ECPD. As it was shown in \cite{weissANDornstein} that `Every transformation is bilaterally deterministic', we can deduce there is an abundance of bilaterally deterministic
     ECPD.

For compactness we omit our results on conservativity of typical measures from this paper. The details can be found in \cite{thesis}.


\textbf{\\ Outline of the Paper.} In Section~\ref{sec:structures} we describe additional machinery needed for proving conservativity and the ergodic theorems. We then prove conservativity in Section~\ref{sec:conserv} and the ergodic theorems in Section~\ref{sec:ergodic}.

\section{ECPS Chains} \label{sec:structures}
In  this Section we introduce dynamical systems we call ECPS Chains, which are (in a sense we will describe soon) equivalent to Extended ECPS. We will then use this equivalence for the proofs of conservativity and the Ergodic Theorems in Sections~\ref{sec:conserv} and \ref{sec:ergodic}.  In this Section we only give an overview of how this equivalence is established, and we leave the proofs of all Lemmas stated in this Section to Appendix~\ref{App:one}.

Let $\PP_0[0,1) \subseteq \NMM$ be the space of measures supported in $[0,1)$, which are either $0$ or probability measures. Let $R:\NMM \to \PP_0[0,1)$ be the restriction of measures to $[0,1)$, i.e., the map $\nu \mapsto R\nu $ where $dR\nu=1_{[0,1)}d\nu$.

 For every  $\mu \in \PP_0[0,1) $ and $0 \leq i \leq p-1$, we define a `zooming in' operation
$$\mu^i=RN\rho_i \mu$$
(note the resemblance to the definition of $\Mext$). Define $LS\subseteq(\PP_0[0,1) \times \Lambdap)^{\ZZ} $, the space of `legal sequences', to be
$$LS=\{(\fullSeq{\mu},\fullSeq{i}): \; \mu_{k+1}=\mu_k^{i_{k+1}} \} $$

For every $n \in \ZZ$ we define the projection maps
$$\hat{i}_n (\fullSeq{\mu},\fullSeq{i})=i_n $$
$$\hat{\mu}_n (\fullSeq{\mu},\fullSeq{i})=\mu_n  $$
We note this is an abuse of notation since $\hat{i}_n$ is also defined on $\NMM \times \Lambdap^{\ZZ}$, but this should not cause any confusion.
\begin{definition}
A distribution $\Pchain$ on $LS$ is called adapted, if for every \hfill \break $j \in \Lambdap$ and almost every $(\fullSeq{\mu},\fullSeq{i}) $,
$$\PP_{\Pchain}\left( \hat{i}_1=j| (\hat{\mu}_n,\hat{i}_n)_{n \leq 0}\right)(\fullSeq{\mu},\fullSeq{i})=\mu_0[j \bp $$
\end{definition}
We remark that if $\Pchain$ is  shift-invariant and adapted then  for every $k \in \ZZ$, $l>0$ and $(j_1,\ldots,j_l) \in \Lambdap ^l $,
\begin{equation}\label{eq:superAdap}
\PP_{\Pchain}\left(\hat{i}_{k+1}=j_1,\ldots,\hat{i}_{k+l}=j_l| (\hat{\mu}_n,\hat{i}_n)_{n \leq k}\right)(\fullSeq{\mu},\fullSeq{i})=\mu_k [j_1,\ldots,j_l)_{p^l}
\end{equation}
For a proof  see \cite{furst}.\begin{definition}
A distribution $\Pchain$ on $LS$ is called an \emph{Ergodic CP chain distribution} (Chain ECPD) if $\Pchain$ is adapted, and invariant and ergodic with respect to the shift operator $\Shift$. If this holds, $(LS,\BB,\Shift,\Pchain)$ is called an \emph{Ergodic CP chain system} (Chain ECPS).
\end{definition}

As in the case of Extended ECPD, it will be convenient to restrict ourselves to Chain ECPD which are non-deterministic:
\begin{definition}
\ We say that a distribution $\Pchain$ on $LS$ is deterministic, if for $\Pchain$ almost every $(\fullSeq{\mu},\fullSeq{i}) $,
$$\PP_{\Pchain}\left(\hat{i}_0=i_0|(\hat{\mu}_n,\hat{i}_n)_{n \leq -1}\right)(\fullSeq{\mu},\fullSeq{i})=1 $$
Otherwise we say that $\Pchain$ is non-deterministic.
\end{definition}
We note that non-deterministic Chain ECPD are supported on the shift-invariant set
$$\PlainXZ=\{(\fullSeq{\mu},\fullSeq{i})\in LS: \; \forall j, \; \fullSeq{i} \not \in E_j   \} $$
An example of a Chain ECPD is given in Subsection~\ref{sub:example}.

In order to describe how one can pass from Chain ECPS to ECPS, we will need to introduce some additional terminology. Let $\II$ denote the set of intervals of the form $[a,b)$ (where $b>a$). Every interval $I=[a,b) \in \II$ can be divided into $p$ disjoint intervals in $\II$ with diameter $\frac{b-a}{p}$. For $0 \leq j \leq p-1$ we define $I^j $ to be the $j-th$ interval. (For example, for $I=[0,1)$, $I^j=[j\bp $).
\begin{definition}
\begin{enumerate}
\item We say that $\minusSeq{I}=(I_n)_{n \leq 0} \subseteq \II$ is compatible with \hfill \break $\minusSeq{i}=(i_n)_{n \leq 0} \subseteq \Lambdap^{\ZZ}$ if for every $n < 0$,
$$I_n^{i_{n+1}}=I_{n+1}$$

\item We say that $\minusSeq{I}$ is well based if $I_0=[0,1)$.
\end{enumerate}
\end{definition}
We note that for every sequence $\minusSeq{i}$, there is a unique sequence $\minusSeq{I}$ which is well based and compatible with $\minusSeq{i}$.

Let us denote the projection of $E_j$ onto the non-positive coordinates by $E_j^{-} $, i.e.
$$E_j^{-}=\{ \minusSeq{i}: \; i_n=j \text{ for all negative enough } n\} $$
We note that if $\minusSeq{i} \not \in E_0^{-}\cup E_{p-1}^{-}$, then the sequence of intervals $\minusSeq{I}$ which is well based and compatible with $\minusSeq{i}$ satisfies $\cup_{n \leq 0}I_n=\RR $.

Let $\HH$ be the group of orientation-preserving homotheties on $\RR$. We note that for every $I,J \in \II$, there is a unique $\rho \in \HH$ such that $\rho(I)=J$. We denote this homothety by $\rho_I^J$.

Now, for every point $(\fullSeq{\mu},\fullSeq{i}) \in (\PP_0[0,1) \times \Lambdap)^{\ZZ}$ we define a measure $\nu \in \NMM$ which in fact depends only on the non-positive coordinates $\nu=\nu(\minusSeq{\mu},\minusSeq{i})$. This measure `preserves the information stored  in the sequence $\minusSeq{\mu}$'.

For a given point $(\fullSeq{\mu},\fullSeq{i})$, we first pick the unique sequence $\minusSeq{I} \subseteq \II$ which is well based and compatible with  $\minusSeq{i}$, and then define for every $n \leq 0$,
 $$\tilde{\mu}_n=N\rho_{I_0}^{I_n} \mu_n$$
\begin{restatable}{lemma}{Lemlamb}
\label{Lemma:lambda} %
 For every $n< 0$, there is a $\lambda=\lambda(n)>0$ such that
$$\tilde{\mu}_{n-1} \restrict_{I_n}=\lambda \tilde{\mu}_n \restrict_{I_n}$$
Moreover there is an $n_0$ such that for all $n < n_0$, $\lambda(n)=1$.
\end{restatable}

Since $\tilde{\mu}_n $ is supported on $I_n$, it follows that for any Borel set $A$, $\mu_n(A)$ increases as $|n|$ increases (at least for $n < n_0 $), and therefore we can define a measure $\nu$ by
$$\nu(A) =\lim_{n \rightarrow - \infty} \tilde{\mu}_n(A)$$

An example for the construction of $\nu$ is given in Subsection~\ref{sub:example}.

Define $\Phi: \PlainXZ \rightarrow \PlainXZtilde$ by
$$\Phi(\fullSeq{\mu},\fullSeq{i})=(\nu(\minusSeq{\mu},\minusSeq{i}),\fullSeq{i}) $$

then we have
\begin{restatable}{lemma}{isomorphism}
Every non-deterministic Chain ECPS $(\PlainXZ,\BB,\Shift,\Pchain)$ is measure-theoretically isomorphic to the non-deterministic Extended ECPS $(\PlainXZtilde,\BB, \Mext, \Phi \Pchain)$.
\end{restatable}
%

For our proofs later on, it will be useful to erase the non-positive coordinates in the spaces we just discussed. Defined $\pi_{-} $ and $\tilde{\pi}_{-} $ on $\PlainXZ $ and $\PlainXZtilde $ respectively, by
$$\pi_{-}(\fullSeq{\mu},\fullSeq{i})=(\minusSeq{\mu},\minusSeq{i})$$
$$\tilde{\pi}_{-}(\nu,\fullSeq{i})=(\nu,\minusSeq{i})$$
and set
$$X=\pi_{-}\PlainXZ \; \; , \; \;   \tilde{X}=\tilde{\pi}_{-}\PlainXZtilde$$

We note that if $\Pchain $ is a non-deterministic chain ECPD, then $\Pminus=\pi_{-} \Pchain$ is an invariant and ergodic distribution on $X$ with respect to the right-shift operator which we denote by $\RightShift \; $.

The map
$$\theta(\minusSeq{\mu},\minusSeq{i})=(\nu(\minusSeq{\mu},\minusSeq{i}),\minusSeq{i}) $$
from $X$ to $\tilde{X}$ is a bijection (the proof is the same as the proof of the previous Lemma) and has the property
\begin{equation} \label{eq:PsiTheta} \tilde{\pi}_{-} \Phi=\theta \pi_{-} \end{equation}

Finally, the map $\tilde{\pi}_{\MM}(\nu,\minusSeq{i})=\nu$ from $\tilde{X} $ to $\NMM$ satisfies
\begin{equation} \label{eq:ProjectionIntertwining} \PM=\tilde{\pi}_{\MM} \tilde{\pi}_{-} \end{equation}
The following diagram summarizes the relations between the spaces and distributions described above (where $\tilde{Q}=\tilde{\pi}_{-}\Pext $) .
\begin{equation*}
\xymatrix{
\left( \PlainXZ, \Pchain \right)  \ar[r]^-{\simeq}_-{\Phi}   \ar[d]^{\pi_{-}}  &  \left( \PlainXZtilde , \Pext \right)\ar[d]^{\tilde{\pi}_{-}} \ar@/^3pc/[dd]^{\hat{\nu}}     \\
  \left(X ,Q \right)   \ar[r]^{\simeq}_{\theta}         &      \left( \tilde{X}, \tilde{Q} \right)  \ar[d]^{\tilde{\pi}_{\MM}} \\
                                &    \left(\NMM  , \hat{\nu}\Pext\right)                   }
\end{equation*}

We now define `translation maps' on $X$ and $\tilde{X}$, which will enable us to prove our assertions regarding the translation maps on $\NMM $ through analogous claims on the translation maps on $X$.
\subsection{Translation Maps}\label{sub:translationOp}
We begin by extending the  definition of the translation maps $\{t^*_k\}_{k \in \ZZ} $ from Subsection~\ref{sub:MainResults} by defining maps $s_k$ on $\Lambdap^{\ZZ_{-}}$, and then defining translation maps $\{T_k \}_{k \in \ZZ}$ on $\tilde{X}$ by
$$T_k(\nu,\minusSeq{i})=(t^*_k \nu,s_k(\minusSeq{i})) $$
 We will then define  analogous `translation maps' on $X$.

Fix some sequence $\minusSeq{i} \not \in E_0^- \cup E_{p-1}^-$ . The sequence of intervals $\minusSeq{I}$ which is compatible with $\minusSeq{i}$ satisfies $\cup I_n= \RR$. Therefore, for a given $k \in \ZZ$ there is some $n_0 \leq 0$ such that $[k,k+1) \subseteq I_{n_0}$. There are indices  $(j_{n_0+1},j_{n_0+2}, \ldots,j_0) \in \Lambdap^{n_0} $ such  that $$[k,k+1)=(\ldots((I_{n_0}^{j_{n_0+1}})^{j_{n_0+2}})\ldots)^{j_0}$$
We now define
$$s_k(\minusSeq{i})_m=\twopartdef{i_m}{m \leq n_0}{j_m}{m > n_0} $$
which completes the definition of $\{T_k \}_{k \in \ZZ} $. We note that $s_k$ is defined so that the sequences $\minusSeq{I} $ and $\minusSeq{J}$ which are well based and compatible with $\minusSeq{i} $ and $s_k(\minusSeq{i}) $ respectively, have the property that for all negative enough $n $, $J_n=I_n-k $. In fact $s_k(\minusSeq{i}) $ is the unique sequence which has this property.  We conclude our discussion of translation maps on $\tilde{X}$ with
\begin{restatable}{lemma}{action}\label{lem:action} %
The maps $\{ T_k \}_{k \in \ZZ} $ define an action of $\ZZ$ on $\tilde{X}$.
\end{restatable}


We now define the analogue of translation maps on $X$.  The group which will act on $X$ is not $\ZZ$ but rather a different group we will now describe. $\Lambdap^{\ZZ_{-}}$ is a group with respect to the addition operation defined by
$$(\minusSeq{i}+\minusSeq{j})_m=i_m+j_m \mod p $$
and $G=\{\minusSeq{i}: \; i_n=0 \text{ for all negative enough } n\} $ is a subgroup, which is the union of the groups
$$G_m=\{\minusSeq{i}: \; i_n=0 \text{ for every } n<m \} $$
defined for all $m \leq 0$. Returning to the definition of $s_k$, we see that $s_k(\minusSeq{i})$ and $\minusSeq{i}$ agree for every $n \leq n_0$, and therefore there is an $a \in G_{n_0+1}$ such that $s_k(\minusSeq{i})=\minusSeq{i}+a $. This fact gives the motivation for the following definition:

For every $m \leq0  $ and $a \in G_m$, the map $S_a:X \to X$ will be defined by
$$S_a(\minusSeq{\mu},\minusSeq{i})=(\minusSeq{\eta},\minusSeq{j}) $$
where
$$\minusSeq{j}=\minusSeq{i}+a $$
 for every $n<m$, $\eta_n=\mu_n$, and $\eta_m,\eta_{m+1}, \ldots \eta_0$ are defined recursively by
$$\eta_m=\eta_{m-1}^{j_m} , \eta_{m+1}=\eta_m^{j_{m+1}}, \ldots \eta_0=\eta_1^{j_0} $$
In other words, $\minusSeq{\eta}$ is the unique sequence of measures with $\eta_n=\mu_n $ for all negative enough $n$ which satisfies $(\minusSeq{\eta},\minusSeq{j}) \in X$.

The relation between the action of $G$ on $X$ and the action of $\ZZ$ on $\tilde{X} $ is given by
\begin{restatable}{lemma}{relation}\label{lem:relation}
For every $(\minusSeq{\mu},\minusSeq{i}) \in X $, the following holds
\begin{enumerate}
\item For every $k \in \ZZ$, there is an $a \in G$ such that
\begin{equation} \label{eq:TranslationIntertwining}\theta S_a (\minusSeq{\mu},\minusSeq{i})=T_k \theta (\minusSeq{\mu}, \minusSeq{i})\end{equation}
\item For every $a \in G$, there is a $k \in \ZZ$ such that Equation~\ref{eq:TranslationIntertwining} holds.
\item Assume that for $a \in G$, $k \in \ZZ$, Equation~\ref{eq:TranslationIntertwining} holds. Then $a=0$ if and only if $k=0$.
\end{enumerate}
\end{restatable}

This can be used to show
\begin{restatable}{lemma}{conservativityequivalence}\label{lem:conservativityequivalence}
The action of $G$ on $X$ is conservative with respect to a probability distribution $Q$ on $X$, if and only if the action of $\ZZ$ on $\tilde{X}$ is conservative with respect to $\theta Q$.
\end{restatable}

Also, the following holds
\begin{restatable}{lemma}{factorconservativity}\label{lem:factorconservativity}
 If the action of $\ZZ$ on $\tilde{X}$ is conservative with respect to $\tilde{Q}$, then the action of $\ZZ$ on $\NMM$ is conservative with respect to $\tilde{\pi}_{\MM}\tilde{Q} $.
\end{restatable}

The conclusion of this Section is that
\begin{proposition}\label{prop:conclusion}
Let $\Pext$ be a non-deterministic $\Mext$ invariant and ergodic probability distribution. If the action of $G$ on $X$ is conservative with respect to $\pi_{-}\Phi^{-1} \Pext $, then the action of $\ZZ$ on $\NMM$ is conservative with respect to $\hat{\nu} \Pext$.
\end{proposition}
This follows immediately from our discussion up to now, since if the action of $G$ on $X$ is conservative with respect to $\pi_{-}\Phi^{-1} \Pext $, then  it follows from Lemma~\ref{lem:conservativityequivalence} that the action of  $\ZZ $ on $\tilde{X}$ is conservative with respect to
$$\theta\pi_{-}\Phi^{-1} \Pext =^{Eq. \ref{eq:PsiTheta}}\tilde{\pi}_{-}\Phi \Phi^{-1} \Pext=\tilde{\pi}_{-} \Pext $$
Which, by Lemma~\ref{lem:factorconservativity} implies that the action of $\ZZ$ on $\NMM$ is conservative with respect to
$$\tilde{\pi}_{\MM} \tilde{\pi}_{-} \Pext=^{Eq. \ref{eq:ProjectionIntertwining}} \hat{\nu} \Pext $$
Therefore, for the proof of Theorem~\ref{th:conservative} it is sufficient to prove
 \begin{proposition}\label{prop:MainLemma}
 Let $Q$ be a non-deterministic ECPD on $X$. Then the action of $G$ on $X$ is  conservative with respect to $Q$.
 \end{proposition}

 The proof of this Proposition is the purpose of the next section.

\subsection{Example}\label{sub:example}
The following is an example of a Chain ECPS, and  the construction of an Extended ECPS from it. We will also use this example in Section~\ref{sec:ergodic}.

\textbf{Chain ECPS.}
         $\lambda_C=(\frac{1}{2}\delta_0+\frac{1}{2}\delta_2))^{\ZZ} $ is a shift-invariant and ergodic measure  on $\Lambda_3^{\ZZ} $.
        Let $\mu_C $ be the standard Cantor measure on $[0,1)$ defined by the property that for every $n>0$ and every $i_1,i_2,\ldots,i_n \in \Lambda_3^n$,
        $$\mu_C[i_1,\ldots,i_n \bpn =\twopartdefotherwise{2^{-n}}{\text{ for all } 1 \leq j \leq n, \; i_j \in \{0,2\}}{0} $$
        We note that for $i=0,2$, $\mu_C^i=\mu_C$, and using this it can be shown that the map
        $$(i_n)_{n \in \ZZ} \mapsto (\mu_n=\mu_C,i_n)_{n \in \ZZ}$$
        defines an isomorphism between $(\Lambda_3^{\ZZ},\Shift,\lambda_C) $ and $(X,\Shift,\delta_{\mu_C}\times \lambda_C) $, and thus
    the Chain distribution $\delta_{\mu_C}\times \lambda_C$ is shift invariant and ergodic. It is also adapted  since
    $$\PP_{\delta_{\mu_C}\times \lambda_C}\left( \hat{i}_1=j| (\hat{\mu}_n,\hat{i}_n)_{n \leq 0}\right)(\fullSeq{\mu},\fullSeq{i})=\lambda_C\left(\{\hat{i}_1=j   \}  \right)= \mu_0[j \bp $$

\textbf{Construction of Extended ECPS.}
        Let $(\fullSeq{\mu},\fullSeq{i})$ be a `typical point', i.e. a point with $\mu_n=\mu_C $ and $i_n \in \{0,2\} $ for every $n \in \ZZ $.  Let us examine  the first stages of the construction in the case where $i_0=0$ and $i_{-1}=2$. Let $\minusSeq{I}$ be the sequence of intervals which is well based and compatible with $\minusSeq{i}$.

        Since $\minusSeq{I}$ is well based, $I_0=[0,1)$  and $\tilde{\mu}_0=\mu_C$.

        Since $i_0=0$, $I_{-1}=[0,3)$ and $$\tilde{\mu}_{-1}=N\rho_{I_0}^{I_{-1}}\mu_{-1}=\mu_C+t_{-2}\mu_C$$
        Since $i_{-1}=2$, $I_{-2}=[-6,3)]$ and $$\tilde{\mu}_{-2}=(\mu_C+t_{-2}\mu_C)+t_6(\mu_C+t_{-2}\mu_C)$$
        Note that $\tilde{\mu}_{-2}\restrict_{I_{-1}}=\tilde{\mu}_{-1}\restrict_{I_{-1}}$ and $\tilde{\mu}_{-1}\restrict_{I_{0}}=\tilde{\mu}_{0}\restrict_{I_{0}}$.

        Continuing in this fashion we will obtain $\tilde{\mu}_n$ which are sums of Cantor measures translated by appropriate integers. We note that measures $\tilde{\mu}_n $ constructed in this fashion will always have $\tilde{\mu}_n[1,2)=0 $. We will use this fact in Section~\ref{sec:ergodic}.

\section{Conservativity}\label{sec:conserv}
In this Section we prove Proposition~\ref{prop:MainLemma}, which will conclude the proof of Theorem~\ref{th:conservative}, as we discussed in the previous Section. As mentioned above, our  proof is an adaptation of a proof of a Lemma from \cite{host}.

\subsection{Conservativity for Increasing Finite Groups of Transformations}
We begin with a general discussion of conservativity for increasing finite groups of transformations.

Let $(Y,\BB,P)$ be a probability space, and for any $n\geq 0$, let $G_n$ be finite groups of measurable maps from $Y$ to $Y$, such that $G_n \subseteq G_{n+1}$.

Additionally, assume that for every $n$, there is a measurable finite partition $\partition^{(n)}=\{R_i^{(n)}\}_{i\leq N_n}$ such that for every $Id \neq T \in G_n$ and $i \leq N_n$, $T(R_i^{(n)})\cap R_i^{(n)}=\emptyset$.

Define measures $P_n=\sum_{T \in G_n} TP$. Clearly for any $A \in \BB$, $P(A)\leq P_0(A) \leq P_{1}(A) \ldots$ and therefore we have Radon-Nikodym derivatives $\phi_n=\frac{dP}{dP_n}$ which are non-increasing $P$ almost everywhere. Additionally $0\leq \phi_n \leq 1$.

Finally, define $\AAA_n$ to be the $\sigma$ algebra of sets invariant under all $T \in G_n$.

\begin{lemma} \label{lem:host}
\begin{enumerate}
\item $P$ is conservative under $G$ if and only if $\phi_n \longrightarrow 0$ $P$ a.e.
\item $\EE_P(f | \AAA_n)(y)= \sum_{T \in G_n}f(Ty)\phi_n(Ty)$, in particular for almost every  $y \in R_i^{(n)}$, \\  $\phi_n(y)=\PP_P(R_i^{(n)} | \AAA_n)(y)$.
\end{enumerate}
\end{lemma}

\begin{proof}
\begin{enumerate}
%

\item Note that for every $n$, $\phi_n>0$ $P$ a.e. Define   $K=\{x: \; \forall n, \; \phi_n(x)>0\}$,\\
  then $P(K)=1$ and $1_K dP_n=\phi_n^{-1}dP$.

    Suppose that $\phi_n$ doesn't tend to zero $P$ a.e., then there exists a $C\subseteq K$ with $P(C)>0$, and $\epsilon>0$, such that for every $n$, $\phi_n>\epsilon$ on $C$. Therefore for every $n$
$$\frac{P(C)}{\epsilon} \geq \int_C \phi_n^{-1} dP = P_n(C)= \int \sum_{T\in G_n} 1_C(Tx) dP(x)$$
The sequence $\sum_{T\in G_n} 1_C(Tx)$ is therefore bounded a.e. , and since it only accepts integer values, it follows easily that there is an $n$ and  $B\subseteq C$ of positive measure s.t. $1_C(Tx)=0$ for every $ T \in G \setminus G_n$ and $x \in B$.

For this $n$, there is an $i\leq N_n$ such that $A=B\cap R_i^{(n)}$ has positive measure, and since for every $T \in G_n \setminus\{ 1_G\}$, $R_i^{(n)} \cap T R_i^{(n)}=\emptyset$, we have $\forall T \neq Id, \; A\cap TA =\emptyset$, and so the action of $G$ isn't conservative with respect to $P$.

Conversely, suppose there is a $B$ of positive measure with
\begin{equation}\label{eq:nullintersect}\forall T \neq Id, \; P(B \cap TB) =0 \end{equation}
Since $P(K)=1$, $\tilde{B}=B \cap K$ has positive measure while \eqref{eq:nullintersect} still holds when we replace $B$ by $\tilde{B}$. The set $A=\tilde{B} \setminus \cup_{T \neq Id} T\tilde{B}$ is then a set of positive measure contained in $K$ with
$$\forall T \neq Id \; A \cap TA =\emptyset$$
 For all $n$,
$$1\geq P(\cup_{T \in G_n}TA)=\sum_{T \in G_n} P(TA)= P_n(A)=\int_A \phi_n^{-1} dP$$
and therefore necessarily for $P$ almost every $x$, $\phi_n(x) \not \rightarrow 0$.
\item $F_f=\sum_{T \in G_n}f(Ty)\phi_n(Ty)$ is clearly invariant under any $T \in G_n$ and therefore $F_f \in L^1(\AAA_n)$. Additionally, for any $A \in \AAA_n$
    \begin{align} \label{eq:local}
    \int 1_A(y)F_f(y)dP(y)&=\int 1_A(y)\sum_{T \in G_n}f(Ty)\phi_n(Ty)dP(y)\\
    &= \int \sum_{T \in G_n} 1_A(Ty)f(Ty)\phi_n(Ty)dP(y) \nonumber\\
    &=\int 1_A(y)f(y)\phi_n(y)dP_n(y)
    =\int 1_A(y)f(y)dP(y) \nonumber
    \end{align}
which proves that indeed $\EE(f | \AAA_n)=F_f$.
\end{enumerate}
\end{proof}

\subsection{Proof of Proposition~\ref{prop:MainLemma}}
We now use the general discussion from the previous Subsection, for the action of the increasing groups $G_n$ on $X$ defined in Subsection~\ref{sub:translationOp} (Note that increasing here is in the sense $G_{-1}\subseteq G_{-2}\subseteq \ldots $).
For every $b=(\ldots,0,0,b_n,b_{n+1},\ldots,b_0) \in G_n $ we define
$$R_b=\{\minusSeq{i}: \;  \hat{i}_n(\minusSeq{i}),\ldots,\hat{i}_0(\minusSeq{i})=b_n,\ldots,b_0\} $$
and
$\partition^{(n)}=\{R_b\}_{b \in G_n} $
This is a finite partition of $X$, and for every $b \in G_n$ and $a \in G_n \setminus \{0\} $,
$$S_a R_b \cap R_b= \emptyset$$
Therefore, according to Lemma~\ref{lem:host}, for every $n \leq 0$ and $\Pminus$ almost every $(\minusSeq{\mu},\minusSeq{i})$, $\phi_n=\frac{d\Pminus}{d\sum_{a \in G_n}S_a\Pminus}$ is given by
$$\phi_n(\minusSeq{\mu},\minusSeq{i})=P_{\Pminus}(\hat{i}_n,\ldots,\hat{i}_0=i_n,\ldots,i_0| \AAA_n)(\minusSeq{\mu},\minusSeq{i}) $$
where $\AAA_n$ is the $\sigma$-algebra invariant under $G_n$. We note that $\AAA_n=(\RightShift)^{-(|n|+1)}\BB $
and therefore we may also write
\begin{align}\label{eq:phin}
\phi_n(\minusSeq{\mu},\minusSeq{i})&=\PP_{\Pminus}\left(\hat{i}_n,\ldots,\hat{i}_0=i_n,\ldots,i_0|(\hat{\mu}_l,\hat{i}_l)_{l <n}\right)(\minusSeq{\mu},\minusSeq{i}) \\
&=\prod_{j=n}^{0} \PP_{\Pminus}\left(\hat{i}_j=i_j|\hat{i}_n,\ldots,\hat{i}_{j-1},(\hat{\mu}_l,\hat{i}_l)_{l <n}\right)(\minusSeq{\mu},\minusSeq{i})  \nonumber \\
&\overset{(*)}=\prod_{j=n}^{0}\PP_{\Pminus}\left(\hat{i}_j=i_j|(\hat{\mu}_l,\hat{i}_l)_{l <j}\right)(\minusSeq{\mu},\minusSeq{i}) \nonumber \\
&\overset{(**)}=\prod_{j=n}^{0} \phi_0 \circ \Shift^{|j|}_{-} (\minusSeq{\mu},\minusSeq{i}) \nonumber
\end{align}

Where $(*)$ follows from the fact that if $\hat{\mu}_{n-1}=\mu_{n-1}$ and $(\hat{i}_n, \ldots,\hat{i}_0)=(i_n,\ldots,i_0) $ then necessarily $$\hat{\mu}_n=\mu_{n-1}^{i_n}=\mu_n, \; \hat{\mu}_{n+1}=\mu_n^{i_{n+1}}=\mu_{n+1}, \; \text{etc.} $$
and $(**)$ follows from the fact that $\Pminus$ is invariant under the right shift.

According to the Ergodic Theorem, for almost every $(\minusSeq{\mu},\minusSeq{i}) $,
$$-\lim_{n \rightarrow -\infty}\frac{1}{|n|}\log \phi_n (\minusSeq{\mu},\minusSeq{i})=-\lim \frac{1}{|n|}\sum_{j=0}^{|n|} \log \phi_0 \circ \Shift^j_{-} (\minusSeq{\mu},\minusSeq{i})=\int -\log \phi_0 d\Pminus $$
We recall that $0 \leq \phi_0 \leq 1$. Also note that there is a set of positive measure on which $\phi_0 <1$ since $Q $ is non-deterministic. Thus $-\int \log \phi_0 d\Pminus>0$, and therefore the set of points $(\minusSeq{\mu},\minusSeq{i})$ for which $\phi_n(\minusSeq{\mu},\minusSeq{i})\rightarrow 0$ is a null set. It now follows from Lemma~\ref{lem:host} that the action of $G$ is conservative with respect to $Q$, thus completing   the proof of Proposition~\ref{prop:MainLemma}. \qedsymbol

\section{Ergodic Theorems} \label{sec:ergodic}
To prove Ergodic theorems for the translation maps described in Section~\ref{sec:conserv}, we will use Hurewicz's Ergodic Theorem and the Chacon-Ornstein Lemma for non-singular transformations. We give a brief summary of what we will need from the theory of non-singular transformations, for reference and more details see \cite{Aaronson}.

\subsection{Non-Singular Transformations}
Let $(X,\BB,\mu)$ be a probability space, and $T:X \to X$ a measurable invertible transformation.

\begin{definition}\label{def:NonSingTrans}
We say that $T$ is non-singular if $T\mu \sim \mu $.
\end{definition}
We note that $T$ is non-singular if and only if $T^{-1}$ is non-singular.
 A non singular-transformation $T$ induces an isometry $U_T:L^1(\mu) \to L^1(\mu)$ through
$$U_T(f)=\frac{dT^{-1}\mu}{d\mu} \cdot f \circ T $$
We say that $T$ is conservative if the $\ZZ$ action defined by $\{T^n: \; n \in \ZZ\} $ is conservative in the sense of Definition~\ref{def:cons}.

We can now state Hurewicz's Ergodic Theorem:
\begin{theorem}[Hurewicz] \label{the:Hur}
Let $(X,\BB,\mu)$ be a probability space, and $T$ a conservative, non-singular transformation. Then for every $f,p \in L^1(\mu) $ with $p>0$
$$\frac{\sum_{k=0}^{n}U_T^k f(x)}{\sum_{k=0}^{n}U_T^k p(x)} \rightarrow \EE_{\mu_p}(\frac{f}{p}|\JJ) $$
for a.e. $x\in X$, where $d\mu_p=pd\mu$ and $\JJ$ is the $\sigma$-algebra invariant under $T$.
\end{theorem}
We will also use the Chacon-Ornstein Lemma
\begin{lemma}[Chacon-Ornstein]\label{lem:ChacOrnst}
For every $f,p \in L^1(\mu) $ with $p>0$
$$\frac{U_T^nf(x)}{\sum_{k=0}^n U_T^kp(x)} \rightarrow 0 $$
for a.e. $x\in X$.
\end{lemma}
A simple calculation show that $U_T^n$ is given by
\begin{equation}\label{eq:IteratedOperator}
U_T^nf=(f \circ T^n)\left(\frac{dT^{-1}\mu}{d\mu}\right)\left( \frac{dT^{-1}\mu}{d\mu} \circ T \right) \ldots \left(\frac{dT^{-1}\mu}{d\mu} \circ T^{n-1} \right)
\end{equation}
\subsection{Proof of Ergodic Theorems}
We now prove the Ergodic Theorems described in Subsection~\ref{sub:MainResults}. Fix some  non-deterministic Extended ECPD $\Pext$. $\Pext$ induces a distribution $\tilde{Q}=\tilde{\pi}_{-} \Pext $ on $\tilde{X}$, and a distribution $Q=\theta^{-1} \tilde{Q} $ on $X$.

The translation maps $\{T_n\}_{n \in \ZZ}$  are not necessarily non-singular with respect to $\tilde{Q}$ . To see this, we first note that the adaptiveness of $\Ptilde$ implies that $\Ptilde$ and $Q$ give the sets
$$\tilde{Y}=\{(\nu,\minusSeq{i}): \; \nu[0,1) \neq 0 \} $$
$$Y=\{(\minusSeq{\mu},\minusSeq{i}): \; \mu_0 \neq 0  \} $$
full measure.
 In Subsection~\ref{sub:example} we described a chain distribution which gives the set $A=\{(\nu,\minusSeq{i}): \; \nu[1,2)=0\}$ full measure while
$$T_1 A \subseteq (\Tilde{Y})^c $$
 and so $T_1 \Ptilde \perp \Ptilde$. Nonetheless,  the set $\tilde{Y}$ and the map $T_1$ induce a map  $T:\tilde{Y}\to \tilde{Y}$ which is non-singular.
 To define $T$, we recall we defined
$$\tau(\nu)=\min\{n \in \NN: \; t_n\nu [0,1)>0\}=\min\{n \in \NN: \; T_n(\nu,\minusSeq{i})\in \tilde{Y}\}$$
\begin{lemma}
For $\Ptilde$ almost every $\nu$, the set $\{n \in \NN: \; t_n\nu [0,1)>0\}$ is non-empty.
\end{lemma}
\begin{proof}
Define
$$g(\nu)=\max\{n \in \ZZ: \; \nu[n,n+1)>0 \} $$
Since almost  every $\nu $ gives  $[0,1) $ positive mass, it is sufficient to show that
$$\tilde{\pi}_{\MM}\tilde{Q}(\{\nu: \; g(\nu)=0 \})=0 $$
Indeed, this must be the case as otherwise, since we know the action of $\ZZ$ on $\NMM$ is conservative with respect to $\tilde{\pi}_{\MM}\tilde{Q}$, there is some $k \in \ZZ \setminus \{ 0\}$ with
$$\tilde{\pi}_{\MM}\tilde{Q}(\{\nu: \; g(\nu)=0 \}\cap t^*_k\{\nu: \; g(\nu)=0 \})>0 $$
but $ t^*_k\{\nu: \; g(\nu)=0 \}=\{\nu: \; g(\nu)=-k \} $ and therefore the intersection above is empty, which gives a contradiction.
\end{proof}

We now define
$$T(\nu,\minusSeq{i})=T_{\tau(\nu)}(\nu,\minusSeq{i})$$

Similarly, we define
$$\tau_{-}(\nu)=\min\{n \in \NN: \; t_{-n}\nu [0,1)>0\}=\min \{n \in \NN: \; T_{-n}(\nu,\minusSeq{i})\in \tilde{Y} \} $$
For almost every $\nu$,
$\tau_{-}(\nu)< \infty$ and it can be verified that the inverse of $T$ is given by
 $$T^{-1}(\nu,\minusSeq{i})=T_{-\tau_{-}(\nu)}(\nu,\minusSeq{i})$$

Our proof for the following lemma is a technical computation presented in Appendix~\ref{l:non_sing}:
\begin{lemma}\label{lem:NonSing}
$T$ is a non-singular transformation and $\deriv \equiv \frac{dT^{-1}\Ptilde}{d\Ptilde} $ is given by
$$\deriv(\nu,\minusSeq{i})=\nu[\tau(\nu),\tau(\nu)+1)$$
\end{lemma}

Note that $\phi$ only depends on the measure coordinate $\nu $. Since $\{T_n: \; n \in \ZZ\}$ is conservative, $T$ is also conservative, and we can use the Hurewicz Ergodic Theorem. To do so we will first calculate $U_T^n f$.

Define $\tau_n(\nu)$ recursively by $\tau_1=\tau $ and
$$\tau_n(\nu)=\tau_{n-1}(\nu)+\tau(t_{\tau_{n-1}(\nu)}^*\nu) $$
Note that if $k\in \NN$ and $k \leq \tau_n(\nu)$, then $\nu[k,k+1)>0$ if and only if there is a $ 1 \leq j \leq n$ such that $k=\tau_j(\nu)$. To keep notation uncluttered, we write $\tau_n$  instead of $\tau_n(\nu)$. Note that
$$\deriv \circ T^n(\nu,\minusSeq{i})=\deriv(t_{\tau_n}^{*}\nu)=t_{\tau_n}^{*}\nu[\tau(t_{\tau_n}^{*}\nu),\tau(t_{\tau_n}^{*}\nu)+1)=
\frac{\nu[\tau_{n+1},\tau_{n+1}+1)}{\nu[\tau_n,\tau_n+1)}$$
and so \eqref{eq:IteratedOperator} gives us
\begin{align*}
U_T^n f(\nu,\minusSeq{i})&=f \circ T^n(\nu,\minusSeq{i}) \cdot  \nu[\tau_1,\tau_1+1)\prod_{j=1}^{n-1}\frac{\nu[\tau_{j+1},\tau_{j+1}+1)}{\nu[\tau_j,\tau_j+1)}\\
&=f \circ T^n(\nu,\minusSeq{i}) \cdot \nu[\tau_n,\tau_n+1)
\end{align*}

We apply the Hurewicz Theorem with $p=1$. Note that
$$\sum_{k=0}^{n-1}U_T^k p(x)=\sum_{k=0}^{n-1} \nu[\tau_k,\tau_k+1)=\nu[0,\tau_{n-1}+1)$$
 so that according to the Hurewicz Theorem,  for a.e. $(\nu,\minusSeq{i})$,
$$\tilde{A}_n^f(\nu,\minusSeq{i}) \equiv \frac{1}{\nu[0,\tau_{n-1}+1)}\sum_{k=0}^{n-1}f \circ T^k(\nu,\minusSeq{i}) \cdot \nu[\tau_k, \tau_k+1) \rightarrow \EE_{\Ptilde}(f|\JJ)(\nu,\minusSeq{i})$$

Define
$$A_m^f(\nu,\minusSeq{i})=\frac{1}{\nu[0,m+1)} \sum_{j=0}^{m} f \circ T_j (\nu,\minusSeq{i}) \cdot \nu[j,j+1)$$
 which is well defined whenever $m \geq \tau_1$. If $\tau_n \leq m<\tau_{n+1}$ then $A_m^f=\tilde{A}_{n+1}^f$ and so for a.e. $(\nu,\minusSeq{i})$,
 $$A_m^f(\nu,\minusSeq{i})\rightarrow \EE_{\Ptilde}(f|\JJ)(\nu,\minusSeq{i}) $$
 Theorem~\ref{th:DiscreteErgodic} is just the special case in which $f(\nu,\minusSeq{i})=f(\nu)$.

 \paragraph{Continuous Ergodic theorem}  Let $f=f(\nu)$ be a bounded measurable function, for $a<b \in \RR$ define
 $$A_a^b(f)(\nu)=\frac{1}{\nu[a,b)}\int_a^b f(t_x^{*}\nu) d\nu(x) $$
We will show that $A_0^x$ converges a.e. (Theorem \ref{th:ContErgodic}). We define $F^f=A_0^1$ and note that it too is bounded and measurable. We also note that if $f=g$ $\Ptilde$ a.e., it is not necessarily true that $F^f=F^g$ $\Ptilde $ a.e. To see this we return to our example of a  non-deterministic ECPD from Subsection~\ref{sub:example}, and note that the set
$$B=\{\nu: \; \forall \epsilon>0, \; \nu[0,\epsilon), \; \nu(1-\epsilon,\epsilon)>0\}$$
 has full measure, but for every $x \in \RR \setminus \ZZ$, $B+x \cap B $ is a null set. Thus we can pick $f=1$ and $g=1_B$, and obtain $F^f=1$ and $F^g=0$.

Nonetheless, we can still apply the ergodic theorem we just proved to $F^f$, to obtain
$$A_0^N(f)(\nu)=A_N^{F^f}(\nu) \rightarrow \EE_{\Ptilde}(F^f|\JJ)(\nu) $$
for a.e. $\nu$. To extend this to any $x \in \RR$, note that if we take $f=p=1$ in the Chacon-Ornstein Lemma, we obtain
$$\lim_{n \rightarrow \infty} \frac{\nu[\tau_n,\tau_n+1)}{\nu[0,\tau_n+1)} \rightarrow 0 $$
and therefore
$$\lim_{x \rightarrow \infty} \frac{\nu[x,\lceil x \rceil)}{\nu[0,\lceil x \rceil)} \rightarrow 0 $$
using this and the fact that for any $0<x<N$,
$$A_0^N(f)(\nu)=\frac{\nu[0,x)}{\nu[0,N)}A_0^x(f)(\nu)+\frac{\nu[x,N)}{\nu[0,N)}A_x^N(f)(\nu)$$
it is not difficult to see that $A_0^x$ converges to the same limit as $A_0^N$ does.

\begin{appendices}

\section{Proofs for Section~\ref{sec:structures} }\label{App:one}

In this Section we give the proofs omitted in Section~\ref{sec:structures}. We begin with Subsection~\ref{sub:preliminaries} in which we present some facts which will be very useful for most of the proofs of Section~\ref{sec:structures}. We then present the proofs themselves in Subsection~\ref{sub:proofs}.

\subsection{Preliminaries} \label{sub:preliminaries}

\emph{Properties of $N$}
In Subsection~\ref{sub:extended} we defined
for every $\nu \in \MM(\RR) \setminus \{0\}$
$$\psi(\nu)=\min\{n\in \NN:\nu[-(n-1),n)>0\}$$
and then defined a normalization map on $\MM(\RR)$ by
$$N \nu=\twopartdef{\frac{\nu}{\nu[-(\psi(\nu)-1),\psi(\nu))}}{\nu \neq 0}{0}{\nu=0}$$
This map has the following properties
\begin{remark} \label{rem:Norm}
\begin{enumerate}
\item For every $\mu$, there is a $\lambda(\mu)>0$ such that $N \mu = \lambda \mu$. Moreover, if
$$\mu_1 \big |_{\left[-\left(\psi(\mu_1)-1\right),\psi(\mu_1)\right)}= \mu_2 \big |_{\left[-\left(\psi(\mu_1)-1\right),\psi(\mu_1)\right)}$$
 then $\lambda(\mu_1)=\lambda(\mu_2) $.
\item For every  $ \lambda>0$, $N(\lambda\mu)=N\mu$. Moreover, if $\rho: \RR \rightarrow \RR$ is a measurable function then (for $\lambda=\lambda(\mu)$)
$$N\rho N \mu=N\rho \lambda \mu=N \lambda \rho \mu=N\rho \mu$$
\end{enumerate}
\end{remark}

In the following we will discuss sequences of measures $(\mu_n)_{n>0}$, with the property that there is an increasing sequence of intervals of the form $I_n=[a_n,b_n)$ with $a_n,b_n \in \ZZ$ such that $\mu_n$ is supported in $I_n$ (in fact we will discuss sequences with non-positive indexes, i.e. $(\mu_n)_{n \leq 0}$, $(I_n)_{n \leq 0}$ etc. but of course this makes no difference). The following Lemma will be useful
\begin{lemma}\label{lem:IncreasingInt}
In the setup described above,

 If for every $n$,  there is an $\eta_n>0$ such that $\eta_n \mu_n \big |_{I_n}=\mu_{n+1} \big |_{I_n}$, then for all large enough $n$, $N \mu_n \big |_{I_n}=N\mu_{n+1}\big |_{I_n} $.
 \end{lemma}
 \begin{proof}
   If for every $n$, $\mu_n=0$, then the claim is trivial. Otherwise, there is an  $n_0 \in \NN$ such that  $\mu_{n_0}(I_{n_0})>0$. Denote $m=\psi(\mu_{n_0}) $.  There is now a large enough $n_1 $ such that
   \begin{equation}
   \label{eq:cover}[-m+1,m)  \subseteq I_{n_1}
   \end{equation}
   For all $n\geq n_1$ It follows from Remark~\ref{rem:Norm} that $\lambda(\eta_n \mu_{n})=\lambda(\mu_{n+1})$ and
   $$N(\mu_{n})=N(\eta_n \mu_{n})=\lambda(\eta_n\mu_n) \mu_{n} $$
   and therefore
   $$N\mu_{n}\restrict_{I_n}=\lambda(\eta_n\mu_n)\mu_n\restrict_{I_n}=\lambda(\mu_{n+1})\mu_{n+1}\restrict_{I_n}=N\mu_{n+1}\restrict_{I_n} $$

 \end{proof}

The following will enable us to interchange limits;
 \begin{remark}[Interchanging limits] \label{InterchangingLimits}
\begin{enumerate}
\item If for all large enough $n$, $ \mu_n \big |_{I_n}=\mu_{n+1}\big |_{I_n} $, then the sequence $\lambda(\mu_n) $ is constant for all large enough $n$, from similar considerations to those used in the previous proof. It follows that for every $A$, $ N\mu_n(A) $ is an eventually increasing sequence, and thus the measure $\lim N\mu_n$ is well defined and
    $$\lim N \mu_n=N \lim \mu_n$$
\item $\mu_n $ converges to $\lim \mu_n $ in the weak topology, and therefore, If $\rho$ is an invertible continuous function whose inverse is also continuous, then $\lim \rho \mu_n=\rho \lim \mu_n $.
\end{enumerate}
\end{remark}

\textbf{Properties of $\HH$. }
We recall that we denoted the set of orientation preserving homotheties by $\HH$, and the intervals of the form $[a,b)$ by $\II$, and defined $\rho_I^J $ to be the unique orientation preserving homothety which takes the interval $I \in \II$ to the interval $J \in \II$.

For any $m\in \ZZ \cup \{\infty\}$,  $(I_n)_{n<m}$  is compatible with $(i_n)_{n <m} $ if for every $n<m$, $I_n=I_{n-1}^{i_n} $. If $I_0=[0,1) $ we say that the sequence is well based.  We record the following properties of $\HH$:
\begin{enumerate}

\item For $I,J,K \in \II$, $\rho_J^K\rho_I^J=\rho_I^K$, since these are both homotheties in $\HH$ which take $I$ to $K$, and therefore uniqueness implies that they are equal.
\item \label{it:PosPreserv} for $\rho \in \HH$, $I \in \II$ and $i \in \Lambdap$, $\rho(I^i)=(\rho(I))^i$. It follows that $\rho_I^J=\rho_{I^i}^{J^i} $, and that for $m \in \ZZ \cup \{\infty \} $, if $(I_n)_{n<m} $ is compatible with $(i_n)_{n<m}$, then so is $(\rho I_n)_{n<m}$.
\end{enumerate}

\subsection{Proofs}\label{sub:proofs}
We now present the proofs of the Lemmas from Section~\ref{sec:structures}.  The proofs will use the observations presented in the previous Subsection extensively. We will not give a reference each time one of these observations is used.

Recall that we defined for every $(\fullSeq{\mu},\fullSeq{i}) \in LS $ measures $\tilde{\mu}_n$ obtained by choosing the sequence of intervals $\minusSeq{I}$ which is well ordered and compatible with $\minusSeq{i}$, and defining $\tilde{\mu}_n=N\rho_{I_0}^{I_n}\mu_n $.

\Lemlamb*

%

\begin{proof}
 Note that since $\mu_n=\mu_{n-1}^{i_n} $, $\mu_n$ and $\rho_{i_n}\mu_{n-1}=\rho_{I_0^{i_n}}^{I_0} \mu_{n-1} $ agree on $I_0=[0,1) $ up to a multiplicative constant, and therefore if we pushforward both measures by $\rho_{I_0}^{I_n} $, they will agree on $I_n$ up to a multiplicative constant. Note that $\rho_{I_0}^{I_n}\mu_n $ is, up to normalization, equal to $\tilde{\mu}_n $, and since
 $$\rho_{I_0}^{I_n}\rho_{I_0^{i_n}}^{I_0}=\rho_{I_0^{i_n}}^{I_n}=\rho_{I_0^{i_n}}^{I_{n-1}^{i_n}}=\rho_{I_0}^{I_{n-1}} $$
 it follows that $\rho_{I_0}^{I_n} \rho_{I_0^{i_n}}^{I_0} \mu_{n-1}=\rho_{I_0}^{I_{n-1}}\mu_{n-1} $, which is equal to $\tilde{\mu}_{n-1} $ up to normalization. Thus we proved that $\tilde{\mu}_{n-1} $ and $\tilde{\mu}_n $ agree on $I_n$ up to a multiplicative constant.

 The fact that for negative enough $n$, $\lambda(n)=1 $, follows from Lemma~\ref{lem:IncreasingInt}.
\end{proof}

The last Lemma enabled us to define a measure $\nu=\nu(\minusSeq{\mu},\minusSeq{i}) $ by $\nu(A)=\lim_{n \rightarrow -\infty}\tilde{\mu}_n(A) $
and $\Phi:\PlainXZ \to \PlainXZtilde $ by
$$\Phi(\fullSeq{\mu},\fullSeq{i})=(\nu(\minusSeq{\mu},\minusSeq{i}),\fullSeq{i}) $$
We then claimed
\isomorphism*
 Recall that we defined $\Mext $ on $\PlainXZtilde $ by
 $$\Mext(\nu,\fullSeq{i})=(N\rho_{i_1}\nu,\Shift(\fullSeq{i})) $$
Note that  $\rho_{i_1}=\rho_{I_1}^{I_0} $, where $I_0,I_1 $ are members of the sequence of intervals $\minusSeq{I} $ which is  well based and compatible with $\minusSeq{i} $.
 \begin{proof}
\textbf{$\Phi$ is a factor map.} For every $(\fullSeq{\mu},\fullSeq{i}) \in LS$,
$$\Phi \circ \Shift (\fullSeq{\mu},\fullSeq{i})=\Phi((\mu_{n+1},i_{n+1})_{n \in \ZZ})=(\nu((\mu_{n+1},i_{n+1})_{n \leq 0}),(i_{n+1})_{n \in \ZZ}) $$
while
\begin{align*}
\Mext \circ \Phi (\fullSeq{\mu},\fullSeq{i})&=\Mext(\nu(\minusSeq{\mu},\minusSeq{i}),\fullSeq{i})\\
&=(N\rho_{i_1} \nu((\mu_n,i_n)_{n \leq 0}),(i_{n+1})_{n \in \ZZ}))
\end{align*}
therefore, we need to prove that
$$\nu((\mu_{n+1},i_{n+1})_{n \leq 0})=N\rho_{i_1} \nu((\mu_n,i_n)_{n \leq 0}) $$

Let $(I_n)_{n \leq 1} $ be the sequence of intervals which is well based and compatible with $(i_n)_{n \leq 1}$. It follows that $(I_{n+1})_{n \leq 0} $ is still compatible with $(i_{n+1})_{n \leq 0}$. As homotheties preserve compatibility, $(J_n)_{n \leq 0}=(\rho_{I_1}^{I_0} I_{n+1})_{n\leq 0}$ is also compatible with $(i_{n+1})_{n \leq 0}$, and it is also well based since $\rho_{I_1}^{I_0}I_1=I_0$.
Therefore,
\begin{equation} \label{eq:helper}
\nu((\mu_{n+1},i_{n+1})_{n \leq 0})=\lim_{n \rightarrow -\infty}N \rho_{J_0}^{J_n}\mu_{n+1}
\end{equation}
Now, as $J_0=[0,1)=I_0 $, and the homothety in $\HH$ which takes $I_{n+1} $ to $J_n=\rho_{I_1}^{I_0}I_{n+1} $ is just $\rho_{I_1}^{I_0}$ by definition,
$$\rho_{J_0}^{J_n}=\rho_{I_0}^{J_n}=\rho_{I_{n+1}}^{J_n}\rho_{I_0}^{I_{n+1}}=\rho_{I_1}^{I_0}\rho_{I_0}^{I_{n+1}}$$
 and so, returning to Eq~\ref{eq:helper} we obtain
 \begin{align*}
 \nu((\mu_{n+1},i_{n+1})_{n<0})&=\lim_{n \rightarrow -\infty}N \rho_{I_1}^{I_0}\rho_{I_0}^{I_{n+1}} \mu_{n+1}=  \lim_{n \rightarrow -\infty}N \rho_{I_1}^{I_0}N\rho_{I_0}^{I_{n+1}} \mu_{n+1}\\
 &=\lim_{n\rightarrow -\infty} N \rho_{I_1}^{I_0} \tilde{\mu}_{n+1}=N\lim_{n \rightarrow -\infty}\rho_{I_1}^{I_0} \tilde{\mu}_{n+1}\\
 &=N\rho_{I_1}^{I_0}\lim_{n \rightarrow -\infty}\tilde{\mu}_{n+1}=N\rho_{i_1}\nu((\mu_n,i_n)_{n \leq 0})
\end{align*}

\textbf{$\Phi$ is a bijection.} An explicit formula for the inverse of $\Phi$ is given by
$$\Phi^{-1}(\nu,\fullSeq{i})=(NR\rho_{I_n}^{I_0}\nu\restrict_{I_n},i_n)_{n \in \ZZ} $$
where $\minusSeq{I}$ is the sequence which is well based and compatible with $\minusSeq{i} $. Thus $\Phi $ is a bijection.


\end{proof}

Recall that for $k \in \ZZ$, we defined $T_k:\tilde{X} \to \tilde{X} $ by
$$T_k(\nu,\minusSeq{i})=(Nt_k\nu,s_k(\minusSeq{i})) $$
where $s_k(\minusSeq{i})$ is the unique sequence which agrees with $\minusSeq{i}$ on all but a finite number of coordinates with the following property:  The sequences of intervals $\minusSeq{I} $ and $\minusSeq{J}$ which are well based and compatible with $\minusSeq{i}$ and $s_k(\minusSeq{i}) $ respectively, satisfy $J_n=I_n-k$  for all negative enough $n$.
\action*
\begin{proof}
We need to prove that for all $l,k \in \ZZ $, $Nt_l  Nt_k=Nt_{l+k} $ and $s_l  s_k=s_{l + k}$.
Since the translation maps $t_k$ are members of
 $\HH$, it follows that for every $\nu\in \NMM $,
$$Nt_lNt_k\nu=Nt_l t_k \nu=Nt_{l+k}\nu$$
If $\minusSeq{I}$, $\minusSeq{J}$ and  $\minusSeq{F} $ are well based and compatible with $\minusSeq{i} $, $s_k(\minusSeq{i})$ and $s_l(s_k(\minusSeq{i})) $ respectively, then for all negative enough $n$, $J_n=I_n-k $ and $F_n=J_n -l $ and therefore  $F_n=I_n-(l+k)$.  Uniqueness now implies that
$$s_{l+k}(\minusSeq{i})=s_l(s_k(\minusSeq{i})) $$
\end{proof}

Recall that we defined a group
$$G=\{\minusSeq{i}: \; i_n=0 \text{ for all negative enough } n \}$$
 and we defined  `translation maps' $\{ S_a \}_{a \in G}$ on $X$ by
$S_a(\minusSeq{\mu},\minusSeq{i})=(\minusSeq{\eta},\minusSeq{j})$
where $\minusSeq{j}=\minusSeq{i}+a$ and $\minusSeq{\eta} $ is the unique sequence of measures with $\mu_n=\eta_n $ for all negative enough $n$, which additionally satisfies $(\minusSeq{\eta},\minusSeq{j}) \in X$.

\relation*

\begin{proof}
\begin{enumerate}
\item For  given $k \in \ZZ $ and  $(\minusSeq{\mu},\minusSeq{i}) $, as $\minusSeq{i}$ and $s_k(\minusSeq{i}) $ disagree on a finite number of coordinates,  there is some $a \in G $ such that $s_k(\minusSeq{i})=\minusSeq{i}+a$. We claim that for this $a$, equality holds in the measure coordinate as well, i.e.
$$\theta S_a(\minusSeq{\mu},\minusSeq{i})=T_k\theta(\minusSeq{\mu},\minusSeq{i})=(Nt_k\nu(\minusSeq{\mu},\minusSeq{i}),s_k(\minusSeq{i})) $$
Let $\minusSeq{I} $ and $\minusSeq{J}$ be the sequences of intervals which are well based and compatible with $\minusSeq{i}$ and $s_k(\minusSeq{i})=\minusSeq{i}+a $, then
\begin{align*}
Nt_k\nu(\minusSeq{\mu},\minusSeq{i})&=Nt_k(\lim_{n \rightarrow - \infty} N \rho_{I_0}^{I_n}\mu_n)=\lim_{n \rightarrow - \infty} Nt_kN\rho_{I_0}^{I_n}\mu_n\\
&=\lim_{n \rightarrow - \infty} Nt_k\rho_{I_0}^{I_n}\mu_n=\lim_{n \rightarrow - \infty}N\rho_{I_0}^{t_kI_n}\mu_n=^{(*)} \nu(S_a(\minusSeq{\mu},\minusSeq{i})) \end{align*}
where $(*)$ follows from the fact that for all negative enough $n$, $J_n=t_k I_n $ and the $n$-th measure coordinate of $S_a(\minusSeq{\mu},\minusSeq{i}) $ is just $\mu_n $, and so
$\hat{\mu}_n\left(S_a (\minusSeq{\mu},\minusSeq{i})\right)=N\rho_{I_0}^{t_kI_n}\mu_n$ .
\item For given $a \in G $ and $(\minusSeq{\mu},\minusSeq{i}) \in X $, it is sufficient to find a $k \in \ZZ$ such that $s_k(\minusSeq{i})=\minusSeq{i}+a$ since by the proof of the former claim, there is a $b \in G$ such that $\theta S_b(\minusSeq{\mu},\minusSeq{i})=T_k \theta(\minusSeq{\mu},\minusSeq{i}) $ and in particular $\minusSeq{i}+a=\minusSeq{i}+b$ which implies $a=b$.

To find such a $k \in \ZZ$, let $\minusSeq{I} $ be the sequence well based and compatible with $\minusSeq{i} $, and define
$$\minusSeq{j}=\minusSeq{i}+a $$
 There is some $n\leq 0$ such that for all $m<n $  , $i_m=j_m$. Define a sequence $\minusSeq{J}$, by requiring that for all $m<n $, $J_m=I_m$, and by recursively requiring that for all $m \geq n$, $$J_m=J_{m-1}^{j_m} $$
   $\minusSeq{J}$ is compatible with $\minusSeq{j}$ but not well based. However, there is a $k \in \ZZ $ such that $t_kJ_0=I_0$, and as homotheties preserve compatibility, $(t_kJ_l)_{l \leq 0}$ is well based and compatible with $\minusSeq{j} $. Since for all negative enough $l$, $t_kJ_l=t_kI_l$, $\minusSeq{j}$ fulfills the property that determines $s_k(\minusSeq{i}) $ uniquely, and therefore
$$s_k(\minusSeq{i})=\minusSeq{j}=\minusSeq{i}+a $$
\item This just follows from the fact that $\minusSeq{i}+a=\minusSeq{i}$ if and only if $a=0$, and similarly $s_k(\minusSeq{i})=\minusSeq{i} $ if and only if $k =0$.
\end{enumerate}
\end{proof}

The following can now be easily shown
\conservativityequivalence*
\begin{proof}
Assume that the action of $\ZZ$ on $\tilde{X}$ is conservative with respect to $\theta Q$. Let $A \subseteq X$ be a Borel set, with $Q(A)>0$, we want to prove that $Q(A \cap (\cup_{a \in G \setminus \{ 0\}}S_aA))>0 $.
$$Q(A \cap (\cup_{a \in G \setminus \{ 0\}}S_aA))=\theta Q(\theta A \cap(\cup_{a \in G \setminus \{ 0\}} \theta S_a A))=\theta Q(\theta A \cap(\cup_{k \in \ZZ \setminus \{ 0\}} T_k \theta A))>0 $$
Where $\theta A \cap(\cup_{k \in \ZZ \setminus \{ 0\}} T_k \theta A)$ has positive measure since we assumed the action of $\ZZ$ on $\tilde{X}$ is conservative with respect to $\theta{Q} $. The other direction can be prove in exactly the same way.
\end{proof}
\factorconservativity*

\begin{proof}
If $A \subseteq \NMM $ is a Borel set with $\tilde{\pi}_{\MM}\tilde{Q}(A)>0$, then
$$\tilde{Q}(\{(\nu,\minusSeq{i}) \in \tilde{X}: \; \nu \in A \})=\tilde{\pi}_{\MM}\tilde{Q}(A)>0$$
and therefore the conservativity in the assumption implies that there is some $k \neq 0 $ such that
\begin{align*}
0<\tilde{Q}(\{(\nu,\minusSeq{i}) \in \tilde{X}: \; \nu \in A \}& \cap T_k \{(\nu,\minusSeq{i}) \in \tilde{X}: \; \nu \in A \})
=\tilde{\pi}_{\MM}\tilde{Q}(A\cap t_k^*A)
\end{align*}
\end{proof}

\section{Proof of Lemma~\ref{lem:NonSing}} \label{l:non_sing}

         Recall that we denote the pushforward of a measure $P $ by $f$  by $fP$ or $dfP$, and the multiplication of $P$ by $f$ by $fdP$.

        It is sufficient to show that
        $$ \frac{dT^{-1}\Ptilde}{d\Ptilde}(\nu,\minusSeq{i})=\nu[\tau(\nu),\tau(\nu)+1)$$
        and then since $\frac{dT^{-1}\Ptilde}{d\Ptilde}>0 $ it follows that $T^{-1}\Ptilde\sim \Ptilde$.

        As in the proof of conservativity, in our calculation of $\frac{dT^{-1}\Ptilde}{d\Ptilde}$ we use the `translation maps' $\{S_a\}_{a \in G}$ on $X$. This can be done since the set of full measure $\tilde{Y}$ can be divided into a countable number of disjoint sets
        $$\tilde{Y}_a=\{ (\nu,\minusSeq{i}): \; T^{-1}(\nu,\minusSeq{i})=\theta \circ S_a \circ \theta^{-1}(\nu,\minusSeq{i}) \}$$

        Fix some  $a \in G$, and pick an $n\leq 0$ so that $a \in G_n $. We recall   that, for $Q_n=\sum_{b \in G_n} S_bQ $ and $\phi_n=\frac{dQ}{dQ_n} $, $\phi_n$ is given by
        \begin{align*}
        \phi_n(\minusSeq{\mu},\minusSeq{i})&=^{\text{Eq~\ref{eq:phin}}}\PP_{\Pminus}(\hat{i}_n,\ldots,\hat{i}_0=i_n,\ldots,i_0|(\hat{\mu}_l,\hat{i}_l)_{l <n})(\minusSeq{\mu},\minusSeq{i})\\
        &=^{\text{Eq~\ref{eq:superAdap}}}\mu_{n-1}[i_n,i_{n+1},\ldots,i_0)_{p^{|n|+1}}
        \end{align*}
        %

        Since $Q_n$ is $S_a$ invariant,
        $$\frac{dS_aQ}{dQ_n}=\phi_n \circ S_{-a}$$
        Additionally, note that for all $(\minusSeq{\mu},\minusSeq{i}) \in Y$, $\phi_n(\minusSeq{\mu},\minusSeq{i})>0$. Therefore
        $$1_YdS_aQ \ll 1_Y dQ_n \sim 1_Y dQ =dQ$$
        and so
        \begin{equation}\label{eq:deriv}
        1_YdS_aQ=\frac{dS_aQ}{dQ_n}\left(\frac{dQ}{dQ_n}\right)^{-1} 1_Y dQ=\frac{\phi_n \circ S_{-a}}{\phi_n} dQ
        \end{equation}

        This can be used to show that for every $n \leq 0$ and $a
        \in G_n$,
        \begin{equation}\label{eq:Tderiv}
        1_{T^{-1}\tilde{Y}_a}dT^{-1} \tilde{Q}= 1_{T^{-1}\tilde{Y}_a}
        \frac{\phi_n\circ \theta^{-1} \circ T}{\phi_n \circ \theta^{-1}}d\tilde{Q}
        \end{equation}
        We leave this computation to the end of the proof. We now obtain $\frac{dT^{-1}\tilde{Q}}{d\tilde{Q}}$ through
        \begin{align*}
        dT^{-1}\tilde{Q}&=\sum_{a \in G}1_{T^{-1}\tilde{Y}_a}dT^{-1} \tilde{Q}=
        \sum_{a \in G} 1_{T^{-1}\tilde{Y}_a}\frac{\phi_n\circ \theta^{-1} \circ T}{\phi_n \circ \theta^{-1}}d\tilde{Q}\\
        &=\frac{\phi_n\circ \theta^{-1} \circ T}{\phi_n \circ \theta^{-1}}d\tilde{Q}
        \end{align*}

        We now compute $\frac{\phi_n\circ \theta^{-1} \circ T}{\phi_n \circ \theta^{-1}}$. Let $(\nu,\minusSeq{i})= \theta(\minusSeq{\mu},\minusSeq{i}) $ be some point in $Y$. and let $\minusSeq{I}$ be the sequence of intervals which is well based and compatible with $\minusSeq{i}$. Then
        $$\phi_n \circ \theta^{-1}(\nu,\minusSeq{i})=\mu_{n-1}[i_n,\ldots,i_0)_{p^{|n|+1}}=\frac{\nu[0,1)}{\nu(I_{n-1})}=\frac{1}{\nu(I_{n-1})} $$
        We now apply the equation above to $T(\nu,\minusSeq{i})$ instead of $(\nu,\minusSeq{i}) $. Since for $T(\nu,\minusSeq{i})$ the compatible sequence $\minusSeq{J}$ satisfies $J_{n-1}=I_{n-1}-\tau $ (for $\tau=\tau(\nu) $) we obtain
        $$ \phi \circ \theta^{-1} \circ T(\nu,\minusSeq{i})=\frac{t^*_{\tau}\nu[0,1)}{t^*_{\tau}\nu(I_{n-1}-\tau)}=\frac{\nu[\tau,\tau+1)}{\nu(I_{n-1})}$$
        Which shows that indeed $\frac{dT^{-1}\Ptilde}{d\Ptilde}(\nu,\minusSeq{i})=\nu[\tau(\nu),\tau(\nu)+1)$.

        We now go back to proving Equation~\ref{eq:Tderiv}. For every $f \in L^1(\tilde{Q})$ we have

        \begin{align*}
        \int 1_{T^{-1}\tilde{Y}_a} \cdot fdT^{-1}\tilde{Q}&=\int 1_{\tilde{Y}_{\alpha}} \cdot (f \circ T^{-1}) d \tilde{Q}\\
        &=\int 1_{\tilde{Y}_a}\cdot (f \circ \theta \circ S_a \circ\theta^{-1}) d\theta Q=
        \int (1_{\tilde{Y}_a}\circ \theta)\cdot(f \circ \theta \circ S_a)dQ\\
        &=\int (1_{\tilde{Y}_a}\circ \theta \circ S_{-a}) \cdot(f \circ \theta )dS_aQ=^{(*)}
        \int (1_{\tilde{Y}_a}\circ \theta \circ S_{-a})\cdot(f \circ \theta ) \cdot 1_Y dS_aQ\\
        &=^{\eqref{eq:deriv}}\int (1_{\tilde{Y}_a}\circ \theta \circ S_{-a})\cdot(f \circ \theta )\cdot (\frac{\phi_n \circ S_{-a}}{\phi_n}) dQ \\
        &=\int (1_{\tilde{Y}_a}\circ \theta \circ S_{-a}\circ \theta^{-1})\cdot f \cdot  (\frac{\phi_n \circ S_{-a}\circ \theta^{-1}}{\phi_n\circ \theta^{-1}}) d\tilde{Q}\\
        &=^{(**)}\int 1_{T^{-1}\tilde{Y}_a} \cdot f \cdot (\frac{\phi_n \circ \theta^{-1} \circ T}{\phi_n\circ \theta^{-1}}) d \tilde{Q}
        \end{align*}
        Equality $(*)$ follows from the fact that
        \begin{equation}\label{eq:shortcut}
        1_{\tilde{Y}_a} \circ \theta \circ  S_{-a}=1_{S_a  \theta^{-1} \tilde{Y}_a}=1_{\theta^{-1}  T^{-1} \tilde{Y}_a}
        \end{equation}
        and
        $$\theta^{-1}T^{-1}\tilde{Y}_a\subseteq \theta^{-1}T^{-1}\tilde{Y} \subseteq \theta^{-1}\tilde{Y} \subseteq Y  $$
        and therefore
        $$(1_{\tilde{Y}_a}\circ \theta \circ S_{-a}) \cdot 1_Y=1_{\tilde{Y}_a}\circ \theta \circ S_{-a} $$

        Equality $(**)$ follows from the fact that
        $$1_{\tilde{Y}_a}\circ \theta \circ S_{-a}\circ \theta^{-1}=^{\text{\eqref{eq:shortcut}}}1_{\theta^{-1} T^{-1} \tilde{Y}_a} \circ \theta^{-1}=1_{ T^{-1} \tilde{Y}_a} $$
        and from the fact that if $(\nu,\minusSeq{i})\in T^{-1}\tilde{Y}_a $ then
        $$\theta^{-1}\circ T(\nu,\minusSeq{i})=S_{-a} \circ \theta^{-1}(\nu,\minusSeq{i}) $$

\section{Notation}\label{Ap:Notation}
We have attempted to use notation and definitions similar to \cite{hochman}, but have made certain changes which are more convenient in our context. The main differences in definitions being:
\begin{enumerate}
\item Our `basic interval' is the unit interval $[0,1)$ and not $(-1,1) $.
\item We allowed the space of normalized measures $\NMM$ and the space of probability measures on the basic interval, to include measure with $\mu[0,1)=0 $. We therefore also extended the definition of the normalization map so that it is well defined for such measures. The reason for this is that the set $\{\mu: \; \mu[0,1)>0 \} $ isn't closed under translations, which is our subject here.
\item We define Extended ECPS in a semi-symbolic manner, replacing the space $\NMM \times (-1,1) $ from \cite{hochman} with  (a subset of) $\NMM \times \Lambdap^{\ZZ}  $.
\end{enumerate}
The following table compares our notation with the notation of \cite{hochman}:

\begin{longtable} {p{80pt} p{80pt} p{200pt} }
        This Paper & \cite{hochman} & remarks \\ \hline
         & & \\
        $N$ or ${\Box}^ *$ & ${\Box}^*$         &   Map normalizing  measures on $\RR $.        \\
        $\NMM$ & $\NMM$         & Space of normalized measures.         \\
        $\PP_0[0,1)$ & $\MM^{\Box}$ & Space of probability measures on the `basic interval'. \\
        $R$ & ${\Box}^{\Box}$ & Restriction of a measure to  the `basic interval'.  \\
        $\PlainXZtilde$ & $\NMM \times (-1,1)$ & Spaces on which Extended ECPS are defined. \\
        $\Mext $ & $M_b$ & zooming in maps on Extended ECPS. \\
        $t_x$ & $T_x$ & Translation by $x$ on $\RR $, or the induced translation on $\MM(\RR) $. \\
        $[i_1,\ldots,i_n \bpn$ & $\DD_{p^n}(x)$ & (for appropriate $x$) the interval $[\sum_{j=1}^n i_j p^{-j},\sum_{j=1}^n i_j p^{-j}+p^{-n}) $. \\
\end{longtable} 

\end{appendices}

\end{document}